\documentclass[english]{article}
\usepackage[T1]{fontenc}
\usepackage[latin9]{inputenc}
\usepackage{geometry}
\usepackage{refstyle}
\usepackage{enumitem}
\usepackage{amsmath}
\usepackage{amsthm}
\usepackage{amssymb}
\usepackage{graphicx}
\usepackage{setspace}
\usepackage{nicefrac}
\usepackage{authblk}
\usepackage{color}
\usepackage{lipsum}

\makeatletter
\def\blfootnote{\xdef\@thefnmark{}\@footnotetext}
\makeatother

\makeatletter

\numberwithin{equation}{section}
\numberwithin{figure}{section}
\theoremstyle{plain}
\newtheorem{thm}{\protect\theoremname}
  \theoremstyle{definition}
  \newtheorem{defn}{\protect\definitionname}
  \theoremstyle{plain}
  \newtheorem{prop}{\protect\propositionname}
  \theoremstyle{plain}
  \newtheorem{lem}{\protect\lemmaname}
  \theoremstyle{plain}
  \newtheorem{cor}{\protect\corollaryname}
\usepackage{dsfont}

\@ifundefined{showcaptionsetup}{}{%
 \PassOptionsToPackage{caption=false}{subfig}}
\usepackage{subfig}

\makeatother

\usepackage{babel}
  \providecommand{\definitionname}{Definition}
  \providecommand{\lemmaname}{Lemma}
  \providecommand{\propositionname}{Proposition}
  \providecommand{\corollaryname}{Corollary}
\providecommand{\theoremname}{Theorem}

\title{Hausdorff Dimension of the Record Set of a Fractional Brownian Motion}
\author[1]{Lucas Benigni}
\author[1]{Cl\'ement Cosco}
\author[1]{Assaf Shapira\thanks{Assaf Shapira acknowledges the support of the ERC Starting Grant 680275 MALIG}}
\author[2]{Kay J\"org Wiese}

\affil[1]{Laboratoire de Probabilit\'es et Mod\`eles Al\'eatoires, Universit\'e Paris Diderot}
\affil[2]{CNRS-Laboratoire de Physique Th\'eorique de l'Ecole Normale Sup\'erieure}

\newcommand{\Addresses}{{
  \bigskip
  \footnotesize
  
  Lucas Benigni, \textsc{Laboratoire de Probabilit\'es et Mod\`eles Al\'eatoires, Universit\'e Paris Diderot, 5 Rue Thomas Mann, 75013 Paris.}\par\nopagebreak
  \textit{E-mail address}:
\texttt{lucas.benigni@math.univ-paris-diderot.fr}

  \medskip

   Cl\'ement Cosco, \textsc{Laboratoire de Probabilit\'es et Mod\`eles Al\'eatoires, Universit\'e Paris Diderot, 5 Rue Thomas Mann, 75013 Paris.}\par\nopagebreak
  \textit{E-mail address}: \texttt{clement.cosco@math.univ-paris-diderot.fr}

  \medskip

  Assaf Shapira, \textsc{Laboratoire de Probabilit\'es et Mod\`eles Al\'eatoires, Universit\'e Paris Diderot, 5 Rue Thomas Mann, 75013 Paris.}\par\nopagebreak
  \textit{E-mail address}: \texttt{assafshap@gmail.com}

  \medskip

  Kay J\"org Wiese, \textsc{CNRS-Laboratoire de Physique Th\'eorique de l'Ecole Normale Sup\'erieure, PSL Research University, Sorbonne Universit\'es, UPMC, 24 rue Lhomond, 75005 Paris, France.}\par\nopagebreak
  \textit{E-mail address}: \texttt{wiese@lpt.ens.fr}

}}

\date{}

\begin{document}

\maketitle

\begin{abstract}
We prove that the Hausdorff dimension of the record set of a fractional
Brownian motion with Hurst parameter $H$ equals $H$.
\end{abstract}

\section{Introduction}
The statistics of records has been studied in both the physics and mathematics literature, see for example
\cite{majumdar2008universal,godreche2017record,feller1966introduction,glick1978breaking,wergen2011record,LeDoussalWiese2008a,arnold1998records}.
The record set (denoted Rec) of a random process $X_{t}$ is the
set of times $s$ at which $X_{s}=\max_{\left[0,s\right]}X_{t}$.
One of the most basic properties of this set is the number of records occurring during a certain
time interval. This problem has been well studied for discrete
processes such as sequences of i.i.d. random variables \cite{arnold1998records,glick1978breaking}
or random walks on $\mathbb{R}$ \cite{andersen1954fluctuations,feller1966introduction}.
However, when considering continuous processes (e.g., the Brownian motion) the question is ill defined.
Indeed, an interval will typically
contain either zero or infinitely many records. In these cases, a
natural way to quantify the size of the record set is to evaluate its Hausdorff
dimension. For the Brownian motion, it is shown in \cite{morters2010brownian} that this dimension is $\frac{1}{2}.$ 

The fractional Brownian motion (fBm) is a continuous Gaussian process $X_{t}$,
depending on a parameter $H\in\left(0,1\right)$ called the Hurst
index. It has expected value $0$ and covariances given by

\[
\left\langle X_{t}X_{s}\right\rangle =\frac{1}{2}\left(\left|t\right|^{2H}+\left|s\right|^{2H}-\left|t-s\right|^{2H}\right).
\]
The fBm is scale-invariant, namely $\left(a^{-H} X_{at}\right)_{t\ge0}$ has the same law as $\left(X_t\right)_{t\ge0}$ for all $a>0$. We emphasize that, even though in general this process could be defined also for $H=1$, we will only consider here $H$ strictly smaller than $1$.

The fractal properties of   fBm have been studied extensively (see
\cite{adler1981geometry,xiao2013fractal}). In this paper, we show
that the Hausdorff dimension of its record set is $H$.
\begin{figure}
\subfloat[]{\includegraphics[scale=0.42]{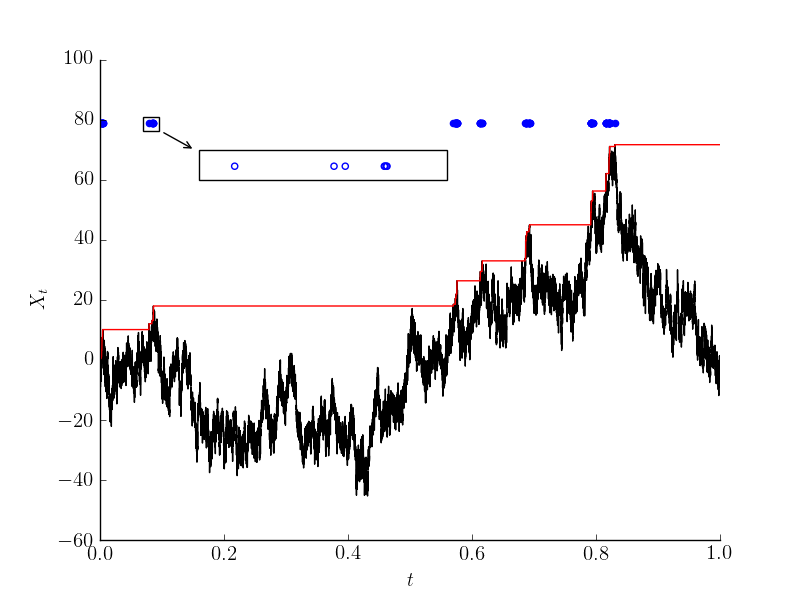}}
\subfloat[]{\includegraphics[scale=0.42]{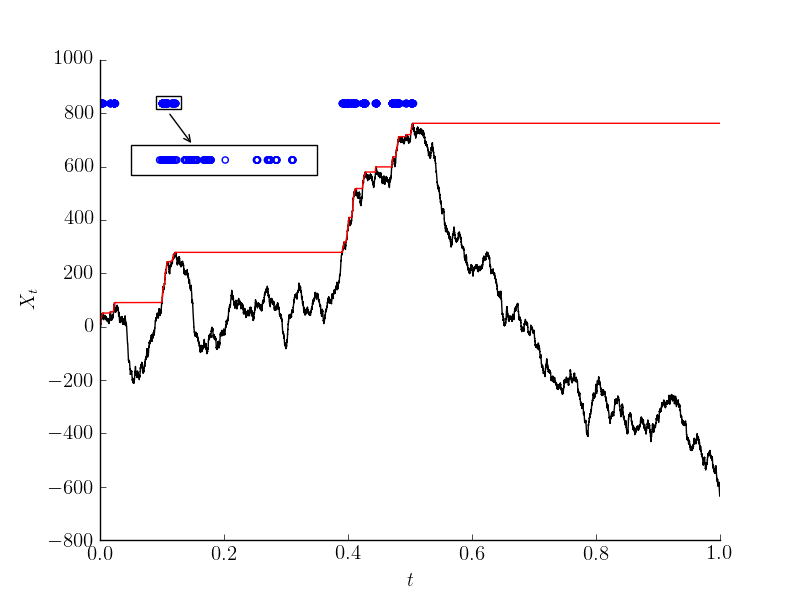}}
\caption{Two simulations of fBm, (a) with Hurst index $\frac{1}{3}$, and (b)
with index $\frac{2}{3}$. The fBm path is in black. In red, for every
time $t$, is the maximum up to time $t$. The blue points represent
the record set. Zooming in one of the blue clusters, one sees the
fractal nature of this set. For generating the fBm we used the algorithm of \cite{DiekerPhD}, for a system of size $2^{25}$.}
\label{fig:fBmswithmaxandrecs}
\end{figure}

\section{Heuristics}

To find the dimension of the record set, first fix a small $\varepsilon>0$
and divide the time interval $\left[0,1\right]$ in $N_{\varepsilon}=\frac{1}{\varepsilon}$
small boxes, each of diameter $\varepsilon$. We will be interested
in finding the number $M_{\varepsilon}$ of boxes in which a record
has occurred. To do so, we first compute the probability to
find a record during the time interval $\left[\left(n-1\right)\varepsilon,n\varepsilon\right]$.
Stated differently, this is the probability that the maximum of
$X_{t}$ in $\left[0,n\varepsilon\right]$ is attained in
the time interval $\left[\left(n-1\right)\varepsilon,n\varepsilon\right]$.
By time reversal symmetry this is the same as the probability
to attain the maximum of $X_t$ during $\left[0,n\varepsilon\right]$ in $\left[0,\varepsilon\right]$. Since the maximum
during the time interval $\left[0,\varepsilon\right]$ scales like $\varepsilon^{H}$,
following \cite{delorme2016perturbative}, we claim that this probability
is controlled by the probability that $\max_{[0,n\varepsilon]} X_{t}$ is of order
$\varepsilon^{H}$. That probability, as shown in \cite{aurzada2011one}, scales like
$\left(\varepsilon^{H}\right)^{\frac{1-H}{H}} = \varepsilon ^{1-H}$. Summing up the argument
so far, we get:
\begin{eqnarray*}
\mathbb{P}\left[\text{Rec}\cap\left[\left(n-1\right)\varepsilon,n\varepsilon\right]\neq\emptyset\right] & = & \mathbb{P}\left[\max_{\left[0,n\varepsilon\right]}X_{t}\text{ is attained during }\left[(n-1)\varepsilon,n\varepsilon\right]\right]\\
 & = & \mathbb{P}\left[\max_{\left[0,n\varepsilon\right]}X_{t}\text{ is attained during }\left[0,\varepsilon\right]\right]\\
 & \approx & \mathbb{P}\left[\max_{\left[0,n\varepsilon\right]}X_{t}\text{ is of order }\varepsilon^{H}\right]\\
 & \approx & \varepsilon^{1-H}.
\end{eqnarray*}
Thus, the expected number of boxes containing
a record scales as $M_{\varepsilon}\approx N_{\varepsilon}\varepsilon^{1-H}=\varepsilon^{-H}$,
suggesting a fractal dimension $H$ of the record set. This scaling
is verified numerically in figure \ref{fig:numerics}, as well as
in \cite{aliakbari2017records}.
\begin{figure}[!h]
\subfloat[\label{fig:H075}]{\hfill{}\includegraphics[scale=0.4]{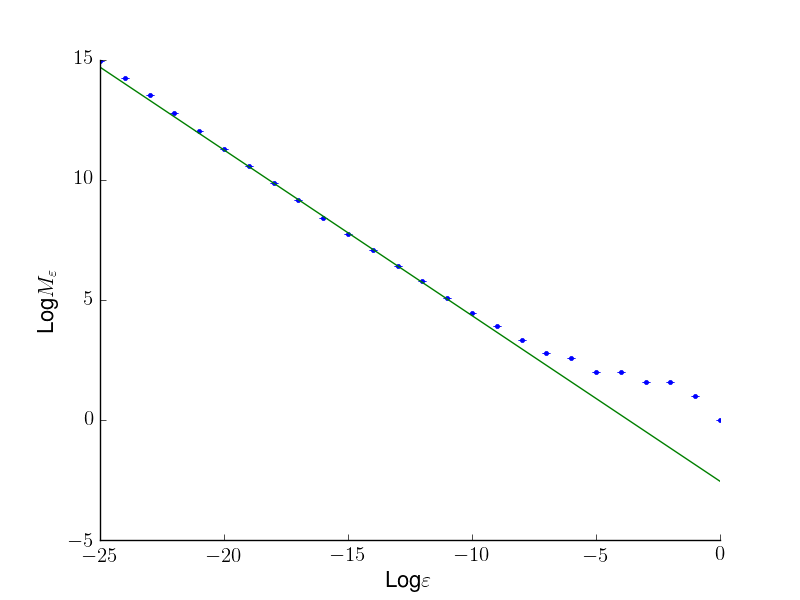}

}\subfloat[\label{fig:DimvsH}]{\hfill{}\includegraphics[scale=0.4]{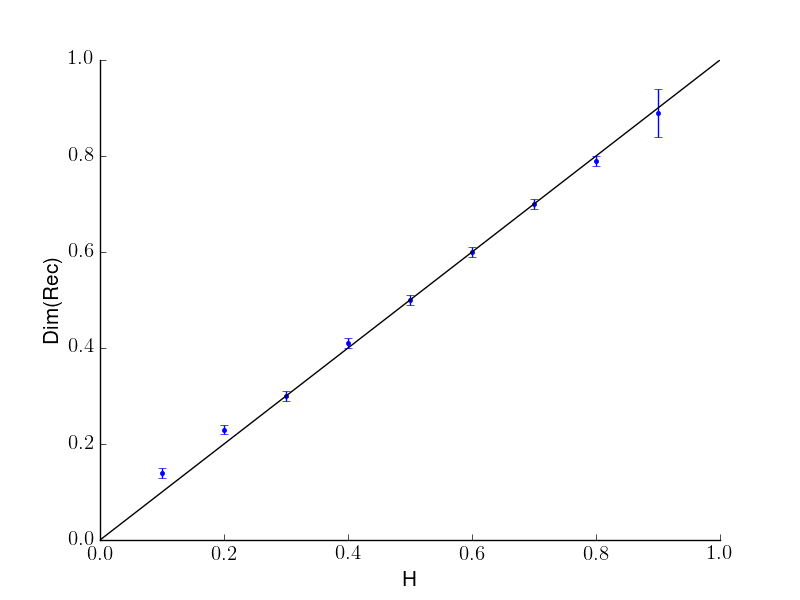}
}
\caption{\label{fig:numerics}In (a) we see a graph showing the relation between
$\log_{2}M_{\varepsilon}$ and $\log_{2}\varepsilon$ for a fBm sample
with $H=\frac{3}{4}$.  The relation is linear with slope $-0.690$
given by linear regression. In (b), we have extracted this slope for
different values of $H$ (shown in blue), averaging over 100 samples. The black line is  
Dim(Rec) $=H$. To generate an fBm we used the algorithm of \cite{DiekerPhD}, for a system of size $2^{25}$.}
\end{figure}

\section{Notation and Presentation of the Result}

We start by presenting some notations and definitions that will be used throughout the proof. In order to give the definition of the Hausdorff dimension we first define the $\alpha$-value of a covering:

\begin{defn}
Let $E$ a metric space, $\mathfrak{E}=\{E_1, E_2\dots\}$ a covering of $E$, and $\alpha\geqslant 0$. Then the $\alpha$\emph{-value} of $\mathfrak{E}$ is
$$\mathcal{S}_\alpha(\mathfrak{E}):=\sum_{i=1}^\infty\mathrm{Diam}(E_i)^\alpha.$$
\end{defn}
We can now define the Hausdorff dimension of a set.
\begin{defn}
Let $E$ a metric space, and for every $\alpha\geqslant 0$ consider the $\alpha$-\emph{Hausdorff measure} of $E$:
$$\mathcal{H}_\alpha(E):=\lim_{\delta \downarrow 0}\inf\left\{S_\alpha(\mathfrak{E}),\,\, \mathfrak{E}\text{ a covering of }E\text{ and }\mathrm{Diam}(E_i)\leqslant \delta\right\}.$$
Then the \emph{Hausdorff dimension} of $E$ is 
$$\mathrm{Dim}(E)=\inf\left\{\alpha\geqslant 0, \mathcal{H}_{\alpha}(E)<\infty\right\}=\sup\left\{\alpha\geqslant 0,\mathcal{H}_\alpha(E)=\infty\right\}.$$
\end{defn}
~\\
Recall   the definition of the record set
$$\mathrm{Rec}=\left\{t\geqslant 0, X_t=\max_{s\in[0,t]}X_s\right\}.$$
The main result we present here is:
\begin{thm} \label{main_theorem}
\begin{equation} \label{eq:dim_rec}
\mathrm{Dim}(\mathrm{Rec})=H \ a.s.
\end{equation}
\end{thm}
Finally, we will prove the following corollary, describing the scaling for the fBm equivalent of the third arcsine law (for results in the physics literature beyond the asymptotics see \cite{SadhuDelormeWiese2017}):

\begin{cor} \label{cor}

For all $\delta>0$ there exists $\varepsilon_0>0$, such that for all positive $\varepsilon<\varepsilon_0$,
\begin{equation} \label{eq:cor}
\varepsilon^{1-H+\delta}\le\mathbb{P}\left(\mathbf{Argmax}_{\left[0,1\right]}X_{t}\in\left[0,\varepsilon\right]\right)\le\varepsilon^{1-H-\delta}.
\end{equation}

\end{cor}

\section{Proof of the Result}
We will follow the proof from \cite{morters2010brownian}, in which the Hausdorff dimension of the record set is found for the (non-fractional) Brownian motion. The main difference comes from the non-Markovian behavior of the fBm, a difficulty that we will control with Lemma \ref{lem:probofrec} below.

First, we will get a lower bound on the record set dimension using Lemma 4.21 of \cite{morters2010brownian}:
\begin{prop}\label{Holder}
Let $f:\,[0,1]\rightarrow \mathbb{R}$, $\alpha$-H\"older continuous, whose maximum is not attained at $0$.
Then the Hausdorff dimension of its record set is greater or equal to $\alpha$.
\end{prop}
For the other bound, we will use the following result, essentially proven in \cite{morters2010brownian}:
\begin{prop}\label{Prop6}
Let $A\subset [0,1]$ be a random set and $\vartheta>0$ such that for all $b>0$, there exists $C_b>0$  and a sequence of positive numbers $\varepsilon_k$, converging to zero and satisfying
\begin{equation}
\forall a\geqslant b,\,\,\,\mathds{P}\left[A\cap [a,a+\varepsilon_k]\neq \emptyset\right]\leqslant C_b\varepsilon_k^{1-\vartheta}.\label{eqassum}
\end{equation}
Then, almost surely,
\begin{equation}\label{eqres1}
\mathrm{Dim}(A)\leqslant \vartheta.
\end{equation}

\end{prop}

For completeness, we present here the proof:
\begin{proof}
Let $b>0$, we will show that $\mathrm{Dim}(A\cap[b,1])\leqslant \vartheta$ using assumption \ref{eqassum}. The result will then follow by the countable stability of the Hausdorff dimension.\\
\indent In order to get an upper bound on the dimension, it is enough to find a family of coverings of $[b,1]\cap A$ with diameter going to zero such that the $\vartheta$-value of each covering is finite. To construct such a covering, consider, for $k\in\mathbb{N}$, 
$$N_k=\sup_{j\in\mathbb{N}}\{b+j\varepsilon_k\leqslant 1\}.$$ 

Denote, for $j=0\dots N_k-1,$ $I_j:=b+[j\varepsilon_k,(j+1)\varepsilon_k]$ and consider the collection of intervals:
$$\mathfrak{I}_k=\{I_j,\,I_j\cap A\ \neq \emptyset\}\cup \{[b+N_k\varepsilon_k,1]\}.$$
\indent Let $j\in\{0,\dots,N_k-1\}$. Taking $a=b+j\varepsilon_k$ in the assumption \ref{eqassum}, we have
\begin{equation}
\mathds{P}[A\cap I_j\neq\emptyset]\leqslant C_b\varepsilon_k^{1-\vartheta}.
\end{equation}  
Therefore the covering $\mathfrak{I}_k$ of $A\cap[b,1]$ has an expected $\vartheta$-value of
\begin{align*}
\mathds{E}\left[\mathcal{S}_\vartheta(\mathfrak{I}_k) \right]=\sum_{j=0}^{N_k-1}\mathds{P}[A\cap {I_j}\neq\emptyset]\varepsilon_k^{\vartheta}+\left(1-b-N_k\varepsilon_k\right)^{\vartheta} \leqslant 2C_b.
\end{align*}
Using Fatou's lemma, we have

\begin{equation}
\mathds{E}\left[\liminf_{k\rightarrow\infty}\mathcal{S}_\vartheta(\mathfrak{I}_k)\right]\leqslant \liminf_{k\rightarrow\infty}\mathds{E}\left[\mathcal{S}_\vartheta(\mathfrak{I}_k)\right]\leqslant 2C_b.
\end{equation}
Hence, the liminf is almost surely finite. In particular, there exists a family of coverings whose diameter is going to zero with bounded $\vartheta$-value, and we can conclude that almost surely
\begin{equation}
\mathrm{Dim}(A\cap[b,1])\leqslant \vartheta.
\end{equation}
\end{proof}

To use this proposition, we will need the following lemma:
\begin{lem} \label{lem:probofrec}
For all $\delta>0$ and $b>0$, there exists a constant $C(b,\delta,H)>0$, such that, for small enough $\varepsilon >0$,

\begin{equation} \forall a\geq b, \quad
\mathbb{P}\left[\mathrm{Rec}\cap\left[a,a+\varepsilon\right]\neq\emptyset\right]\le C\,\varepsilon^{1-H-\delta}.
\end{equation}
\end{lem}

\begin{proof}
We begin by introducing two inequalities concerning the supremum of $X$:
\begin{enumerate}[label = (\roman*)]
\item For all $\delta'>0$, there exists a constant $M=M(\delta,H)>0$, such that, for small enough $u>0$:
\begin{equation} \label{eq:ineq1}
\mathbb{P} [ \sup_{0\leq t \leq 1} X_t \leq u] \leq  M \, u^{\frac{1-H-\delta}{H}} .
\end{equation} 
\item There exists a constant $M'=M'(H)>0$, such that, for large enough $v$, we have
\begin{equation} \label{eq:ineq2}
\mathbb{P} [ \sup_{0\leq t \leq 1} X_t > v] \leq M' v^{1/H} \Psi(v),
\end{equation}
where $\Psi(v) = \mathbb{P}(N>v)$ for $N$ a standard normal random variable. 
\end{enumerate}

The first inequality is a weak consequence of corollary 2 in \cite{aurzada2011one}. The statement of the second inequality can be found in Theorem D.4. of \cite{piterbarg2012asymptotic} (see also \cite{adler2009random}).

Let $\delta,$ $\varepsilon$ and $\theta$ be three positive real numbers, with $\theta$ satisfying $\theta <H$. By time reversal symmetry, the process $t \mapsto \tilde{X}_t = X_{a+\varepsilon-t} - X_{a+\varepsilon}$ is again a fractional Brownian motion starting at $0$ with Hurst index $H$. Hence,
\begin{align*}
\mathbb{P}\left[\mathrm{Rec}\cap [a,a+\varepsilon]\neq\emptyset\right] &= \mathbb{P}\left[\sup_{[0,a+\varepsilon]} \tilde{X}_t = \sup_{[a,a+\varepsilon]} \tilde{X}_t \right] = \mathbb{P}\left[\sup_{[0,a+\varepsilon]} {X}_t = \sup_{[0,\varepsilon]} {X}_t \right].
\end{align*}
Decomposing this last term into the two terms
\[ A_\varepsilon = \mathbb{P}\left[\sup_{[0,a+\varepsilon]} {X}_t \leq \varepsilon^{H-\theta}, \sup_{[0,a+\varepsilon]} {X}_t = \sup_{[0,\varepsilon]} {X}_t \right],\]
\[ B_\varepsilon = \mathbb{P}\left[\sup_{[0,a+\varepsilon]} {X}_t > \varepsilon^{H-\theta} , \sup_{[0,a+\varepsilon]} {X}_t = \sup_{[0,\varepsilon]} {X}_t \right],\]
and using the scaling invariance of $X$, we get that
\begin{align}\label{eq:Maj_A}
A_\varepsilon &\leq \mathbb{P}\left[ \sup_{[0,a+\varepsilon]} \frac{{X}_t}{(a+\varepsilon)^{H}}  \leq \frac{\varepsilon^{H-\theta}}{(a+\varepsilon)^{H}} \right] = \mathbb{P}\left[\sup_{[0,1]} {X}_t \leq \frac{\varepsilon^{H-\theta}}{(a+\varepsilon)^{H}} \right].
\end{align}
Therefore, for small enough $\varepsilon$, we can apply inequality \ref{eq:ineq1} with a positive parameter $\delta'<1-H$:
\begin{align*}
A_\varepsilon & \leq M \, \frac{\varepsilon^{1-H-\delta'}}{(a+\varepsilon)^{1-H-\delta'}} \, \varepsilon^{-\theta(1-H-\delta')/H} \leq C_1 \, \varepsilon^{1-H-\delta},
\end{align*}
where we now fix $\theta$ and $\delta'$ sufficiently small, chosen to verify
\[
\delta = -\delta' - \theta \frac{1-H-\delta'}{H},\]
and where, recalling that $b\leq a$, $C_1$ is defined as:
\[C_1(b,\delta,H) = M(\delta,H) \, b^{-1+H+\delta'}.\]

We then bound $B_\varepsilon$, using again the scaling invariance and applying \ref{eq:ineq2} for $\varepsilon$ small enough:
\begin{align*}
B_\varepsilon =  \mathbb{P}\left[\sup_{[0,\varepsilon]} {X}_t >  \varepsilon^{H-\theta}, \sup_{[0,a+\varepsilon]} {X}_t = \sup_{[0,\varepsilon]} {X}_t \right] &\leq \mathbb{P}\left[\sup_{[0,1]} {X}_t > \epsilon^{-\theta} \right]\\
& \leq M'(H) \, \varepsilon^{-\theta/H} \Psi(\varepsilon^{-\theta})\\
& \leq C_2 \, \varepsilon^{1-H-\delta},
\end{align*}
with $C_2=C_2(\delta,H)$ and where the last inequality is a consequence of the rapid decay of $\Psi$ as $\varepsilon$ tends to $0$. Summing the bounds over $A_\varepsilon$ and $B_\varepsilon$ concludes the proof of the lemma.
\end{proof}
Putting everything together, we are ready to prove Theorem \ref{main_theorem}.

\begin{proof}[Proof of Theorem \ref{main_theorem}.] We first prove the lower bound using Proposition \ref{Holder}. Indeed, the sample-paths of the fractional Brownian motion are almost surely $\alpha$-H\"older continuous for any $\alpha<H$ (see Theorem 3.1 of \cite{decreusefond1999stochastic}). Hence, we get that
$$\mathrm{Dim}(\mathrm{Rec})\geqslant \alpha$$
for any $\alpha<H$. The lower bound follows letting $\alpha$ go to $H$.\\
\indent In order to get the upper bound, we use Proposition \ref{Prop6} combined with Lemma \ref{lem:probofrec} to find:
$$\mathrm{Dim(Rec)}\leqslant H+\delta$$
for all $\delta>0$.
The upper bound follows letting $\delta$ go to zero. 
\end{proof}

We can now give a proof of Corollary \ref{cor}:
\begin{proof}
By the time reversibility property of the fractal Brownian motion, we can see that (cf. proof of Lemma \ref{lem:probofrec}):
\[
\mathbb{P}\left[\mathrm{Argmax}_{\left[0,1\right]}X_{t}\in\left[0,\varepsilon\right] \right]= \mathbb{P}\left[\mathrm{Rec}\cap [1-\varepsilon,1]\neq\emptyset\right].
\]

Therefore, the upper bound in the inequality \ref{eq:cor} is a direct consequence of Lemma \ref{lem:probofrec} (taking $\frac{\delta}{2}$ in order to absorb the constant $C$).

For the lower bound, we need to show that for all $\delta>0$, there exists $\varepsilon_0>0$ such that
\begin{equation} 
\forall \varepsilon < \varepsilon_0,\quad\varepsilon^{1-H+\delta}\le \mathbb{P}\left[\mathrm{Rec}\cap [1-\varepsilon,1]\neq\emptyset\right].
\end{equation}

Reasoning by contradiction, let $\delta >0$ and $(\varepsilon_k)_{k\geq 0}$ such that $\varepsilon_k \to 0$ and
\begin{equation} \label{eq:absurd_ineq}
\mathbb{P}\left[\mathrm{Rec}\cap [1-\varepsilon_k,1]\neq\emptyset\right]< \varepsilon_k^{1-H+\delta}.
\end{equation}
Let $b>0$, $a\geq b$, and $\varepsilon'_k$ to be chosen later on. Consider the rescaled process $t\to Y_t= \left({a+\varepsilon'_k}\right)^H X_{({a+\varepsilon'_k})t}$. By scaling invariance,  $Y$ is a fractional Brownian motion of Hurst index $H$ whose record set $\mathrm{Rec}(Y)$ on $[0,1]$ is the rescaled record set of $X$. Hence,
\begin{align*}
\mathbb{P}\left[\mathrm{Rec}(X)\cap [a,a+\varepsilon'_k]\neq\emptyset\right] &= \mathbb{P}\left[\mathrm{Rec}(Y)\cap [1 - \frac{\varepsilon'_k}{a+\varepsilon'_k},1]\neq\emptyset\right]\\
& \leq \mathbb{P}\left[\mathrm{Rec}\cap [1 - \frac{\varepsilon'_k}{b+\varepsilon'_k},1]\neq\emptyset\right].
\end{align*}
Choosing $\varepsilon'_k=\frac{b\,\varepsilon_k}{1-\varepsilon_k}$, so that $\frac{\varepsilon'_k}{b+\varepsilon'_k}=\varepsilon_k$, \ref{eq:absurd_ineq} yields:
\[
\mathbb{P}\left[\mathrm{Rec}\cap [a,a+\varepsilon'_k\right] \leq b^{-1+\delta+H} \, {\varepsilon'}_k^{1-(H-\delta)}.
\]
This is exactly the assumption of Proposition \ref{Prop6}, thus
\[\mathrm{Dim}(\mathrm{Rec}) \leq H-\delta,\]
in contradiction with Theorem \ref{main_theorem}.
\end{proof}

\section{Further Questions}
There are various topics for further research concerning the record statistics of continuous processes. For example, one may study the duration of the longest record or the waiting time for a first record to occur after some fixed positive time. It could also be interesting to study non-Gaussian or non-stationary processes. Another question would be to extend the study of records to fields of higher dimensions (both in space and in time), given an appropriate order on these spaces.

\section*{Acknowledgments}

We thank Mathieu Delorme for providing the python code generating the fBm samples used in the numerical verification. We also thank Francis Comets for careful reading and helpful comments.

\bibliographystyle{plain}
\bibliography{dim_of_records_in_fBm}

\Addresses

\end{document}